\documentclass[11pt]{amsart}
\usepackage{amsmath}
\usepackage{amssymb}
\usepackage{amsfonts}

\newtheorem{theorem}{Theorem}[section]
\theoremstyle{plain}

\newtheorem{corollary}[theorem]{Corollary}

\numberwithin{equation}{section}
\input{tcilatex}

\begin{document}
\title[An observation on initially $\kappa $-compact spaces ]{An observation
on initially $\kappa $-compact spaces }
\author{\c{C}etin Vural }
\address{Gazi Universitesi, Fen Fakultesi, Matematik Bolumu, 06500
Teknikokullar, Ankara, Turkey}
\email{cvural@gazi.edu.tr}
\date{}
\thanks{The author would like to thank S\"{u}leyman \"{O}nal for fruitful
discussions and Hasan G\"{u}l for the first draft of manuscript, and the
referee for valuable remarks.}

\begin{abstract}
In \cite{Chaber}, Chaber has proved that countably compact spaces with a
quasi-$G_{\delta }$-diagonal are compact. We prove that initially $\kappa $%
-compact spaces with a quasi-$G_{\kappa }$-diagonal are compact, for any
infinite cardinal $\kappa .$
\end{abstract}

\maketitle

\section{Introduction and Terminology}

Chaber, in \cite{Chaber} has proved that countably compact spaces with a
(quasi-) $G_{\delta }$-diagonal are compact. We observe that this result may
be generalised by using any infinite cardinal instead of the first infinite
cardinal $\omega .$ For that purpose, we regard countable compactness as
initial $\omega $-compactness and extend it naturally to initial $\kappa $%
-compactness, and the quasi-$G_{\delta }$-diagonal property to quasi-$%
G_{\kappa }$-diagonal property which allows us to conclude that initially $%
\kappa $-compact spaces with a quasi-$G_{\kappa }$-diagonal are still
compact, for any infinite cardinal $\kappa .$

Throughout this paper, $\kappa $ is an infinite cardinal, $\omega $ is the
first infinite ordinal and cardinal, and $X$ is a topological space. Let us
recall some basic definitions. If $A\subseteq X$ and $\mathcal{H}$ is a
family of subsets of $X,$ then \textit{the star of }$\mathcal{H}$ \textit{%
about }$A$ is denoted by $st\left( A,\mathcal{H}\right) =\bigcup \left\{
H\in \mathcal{H}:H\cap A\neq \emptyset \right\} .$ For $x\in X,$ we write $%
st\left( x,\mathcal{H}\right) $ instead of $st\left( \left\{ x\right\} ,%
\mathcal{H}\right) .$ A transfinite sequence $\left\{ \mathcal{O}_{\alpha
}:\alpha \in \kappa \right\} $ of collections of open subsets of $X$ is said
to be a \textit{quasi-}$\mathit{G}_{\kappa }$\textit{-diagonal sequence for }%
$X$ if for each $x,y\in X$ with $x\neq y,$ there exists an $\alpha \in
\kappa $ with $x\in \bigcup \left\{ O:O\in \mathcal{O}_{\alpha }\right\} $
and $y\notin st\left( x,\mathcal{O}_{\alpha }\right) .$ If each $\mathcal{O}%
_{\alpha }$ is a cover of $X,$ then the quasi-$G_{\kappa }$-diagonal
sequence $\left\{ \mathcal{O}_{\alpha }:\alpha \in \kappa \right\} $ is
called a $\mathit{G}_{\kappa }$\textit{-diagonal sequence for }$X.$ For $%
\kappa =\omega ,$ a (quasi-)$G_{\kappa }$-diagonal sequence is called a
(quasi-)$G_{\delta }$-diagonal sequence.

It is said that $X$ has a $\mathit{G}_{\kappa }$\textit{-diagonal} if the
diagonal $\Delta _{X}=\left\{ \left( x,x\right) :x\in X\right\} $ of $X$ is
a $G_{\kappa }$-set in $X\times X,$ (that is, the set $\Delta _{X}$ is the
intersection of $\kappa $-many open sets of $X\times X$). In \cite{Ceder},
it was shown that $X$ has a $G_{\delta }$-diagonal if and only if $X$ has a $%
G_{\delta }$-diagonal sequence. It can easily be seen that $X$ has a $%
G_{\kappa }$-diagonal if and only if $X$ has a $G_{\kappa }$-diagonal
sequence. \textit{The diagonal number} $\Delta \left( X\right) $ of a space $%
X$ is the smallest infinite cardinal $\kappa $ such that the diagonal $%
\Delta _{X}$ of $X$ is the intersection of $\kappa $-many open sets of $%
X\times X.$ Thus, if $X$ has a $G_{\kappa }$-diagonal, then $\Delta \left(
X\right) \leq \kappa .$

Recall that a topological space $X$ is said to be \textit{initially }$\kappa 
$\textit{-compact} if every open cover of $X$ of cardinality not exceeding $%
\kappa $ has a finite subcover. For $\kappa =\omega ,$ initial $\omega $%
-compactness is equivalent to countable compactness. For a survey of initial 
$\kappa $-compactness see \cite{Stephenson}.\bigskip

The cardinality of a set $A$ is denoted by $\left\vert A\right\vert $.
Recall that the the \textit{weight} of $X$ is denoted by $w(X)$ and is
defined as the smallest cardinal number of the form $\left\vert \mathcal{B}%
\right\vert ,$ where $\mathcal{B}$ is a base for $X.$

\section{Main Result}

The following statement generalizes Chaber's theorem in \cite{Chaber}.

\begin{theorem}
An initially $\kappa $-compact space with a quasi $G_{\kappa }$-diagonal is
compact.
\end{theorem}

\begin{proof}
Let $X$ be an initially $\kappa $-compact space, and let $\left\{ \mathcal{O}%
_{\alpha }:\alpha \in \kappa \right\} $ be a quasi-$G_{\kappa }$-diagonal
sequence for $X$. Suppose that $X$ is not compact. Let $\mathcal{M}$ be a
maximal open cover of $X$ without any finite subcover, (that is, $\mathcal{M}
$ is an open cover of $X$ without any finite subcover, and the cover $%
\mathcal{M\cup }\left\{ O\right\} $ has a finite subcover for any open
subset $O$ of $X$ with $O\notin \mathcal{M}$). Since the space $X$ is
initially $\kappa $-compact, $\mathcal{M}$ has no subcover of cardinality at
most $\kappa $. We claim that for each $x$ in $X$ there exists an $\alpha
\left( x\right) \in \kappa $ such that $x\in \bigcup \mathcal{O}_{\alpha
\left( x\right) }$ and $st\left( x,\mathcal{O}_{\alpha \left( x\right)
}\right) \in \mathcal{M}$. To prove this claim, suppose that $st\left( x,%
\mathcal{O}_{\alpha }\right) \notin \mathcal{M}$ for all $\alpha \in \kappa $
satisfying $x\in \bigcup \mathcal{O}_{\alpha }.$ Let $J=\left\{ \alpha \in
\kappa :x\in \bigcup \mathcal{O}_{\alpha }\right\} $. Since $st\left( x,%
\mathcal{O}_{\alpha }\right) $ is an open subset of $X$, the maximality of $%
\mathcal{M}$ gives us a finite subcover $\mathcal{H}_{\alpha }$ of the open
cover $\mathcal{V}_{\alpha }=\mathcal{M\cup }\left\{ st\left( x,\mathcal{O}%
_{\alpha }\right) \right\} ,$ for all $\alpha \in J.$ So, we have a finite
subfamily $\mathcal{W}_{\alpha }$ of $\mathcal{M}$ such that $\mathcal{H}%
_{\alpha }=\mathcal{W}_{\alpha }\mathcal{\cup }\left\{ st\left( x,\mathcal{O}%
_{\alpha }\right) \right\} ,$ for each $\alpha \in J.$ Since $\mathcal{H}%
_{\alpha }$ is a cover of $X$ and $\left\{ \mathcal{O}_{\alpha }:\alpha \in
\kappa \right\} $ is a quasi-$G_{\kappa }$-diagonal sequence for $X,$ we
have $X\backslash \left\{ x\right\} \subseteq \bigcup\nolimits_{\alpha \in
J}\left( \bigcup \mathcal{W}_{\alpha }\right) .$ Take an $M\in \mathcal{M}$
with $x\in M.$ It is clear that the family $\left\{ W:W\in \mathcal{W}%
_{\alpha },\alpha \in J\right\} \cup \left\{ M\right\} $ is a subcover of $%
\mathcal{M}$ and its cardinality is at most $\kappa $. But this contradicts
the fact $\mathcal{M}$ has no subcover of cardinality at most $\kappa $.
Hence, our claim is true. So, choose an $\alpha \left( x\right) \in \kappa $
such that $x\in \bigcup \mathcal{O}_{\alpha \left( x\right) }$ and $st\left(
x,\mathcal{O}_{\alpha \left( x\right) }\right) \in \mathcal{M}$, for each $%
x\in X.$ Let $Y_{\alpha }=\left\{ x\in X:\alpha \left( x\right) =\alpha
\right\} .$ Obviously, $X=\bigcup\nolimits_{\alpha \in \kappa }Y_{\alpha }.$ 
\newline
Now, we claim that $Y_{\alpha }$ is covered by a finite subfamily of $%
\mathcal{M}$, for each $\alpha \in \kappa $. Take an $\alpha \in \kappa .$
We have two cases:\newline
Case 1: Suppose $st\left( Y_{\alpha },\mathcal{O}_{\alpha }\right) \in 
\mathcal{M}$. Then $Y_{\alpha }$ is covered by $\left\{ st\left( Y_{\alpha },%
\mathcal{O}_{\alpha }\right) \right\} .$\newline
Case 2: Suppose $st\left( Y_{\alpha },\mathcal{O}_{\alpha }\right) \notin 
\mathcal{M}$. In this case, by the maximality of $\mathcal{M}$, we have a
finite subfamily $\mathcal{S}$ of $\mathcal{M}$ such that $X=st\left(
Y_{\alpha },\mathcal{O}_{\alpha }\right) \cup \left( \bigcup \mathcal{S}%
\right) .$ Now, we claim that if $Y_{\alpha }\backslash \left( \bigcup 
\mathcal{S}\right) \neq \emptyset ,$ then there exists a finite subset $%
\left\{ x_{0},x_{1},...,x_{n}\right\} $ of $Y_{\alpha }$ such that $%
Y_{\alpha }\backslash \left( \bigcup \mathcal{S}\right) \subseteq
\bigcup\nolimits_{i=0}^{n}st(x_{i},\mathcal{O}_{\alpha }).$ Indeed, take a
point $x_{0}\in Y_{\alpha }\backslash \left( \bigcup \mathcal{S}\right) .$
If $Y_{\alpha }\backslash \left( \bigcup \mathcal{S}\right) \backslash
st(x_{0},\mathcal{O}_{\alpha })\neq \emptyset ,$ we can choose a point $%
x_{1}\in Y_{\alpha }\backslash \left( \bigcup \mathcal{S}\right) \backslash
st(x_{0},\mathcal{O}_{\alpha }).$ If we can choose inductively a point $%
x_{n}\in Y_{\alpha }\backslash \left( \bigcup \mathcal{S}\right) \backslash
\bigcup\nolimits_{m<n}st(x_{m},\mathcal{O}_{\alpha })$, for each $n\in 
\mathcal{\omega },$ then we obtain a sequence $\left\{ x_{n}:n\in \mathcal{%
\omega }\right\} $ in the closed subspace $X\backslash \left( \bigcup 
\mathcal{S}\right) $ of $X$. Since initial $\kappa $-compactness is
hereditary with respect to closed subsets and every initially $\kappa $%
-compact space is countably compact, the sequence $\left\{ x_{n}:n\in 
\mathcal{\omega }\right\} $ has an accumulation point $x$ in $X\backslash
\left( \bigcup \mathcal{S}\right) .$ Since $X=st\left( Y_{\alpha },\mathcal{O%
}_{\alpha }\right) \cup \left( \bigcup \mathcal{S}\right) ,$ we have $x\in
st\left( Y_{\alpha },\mathcal{O}_{\alpha }\right) .$ So, there exists $O\in 
\mathcal{O}_{\alpha }$ with $x\in O.$ Note that $\left\vert O\cap \left\{
x_{n}:n\in \mathcal{\omega }\right\} \right\vert \leq 1.$ But this
contradicts the fact that $x$ is an accumulation point of the sequence $%
\left\{ x_{n}:n\in \mathcal{\omega }\right\} .$ So, it must be $Y_{\alpha
}\backslash \left( \bigcup \mathcal{S}\right) \backslash \left(
\bigcup\nolimits_{i=0}^{n}st(x_{i},\mathcal{O}_{\alpha })\right) =\emptyset $
for some $n\in \mathcal{\omega }$. Therefore the family $\left\{ st\left(
x_{0},\mathcal{O}_{\alpha }\right) ,st\left( x_{1},\mathcal{O}_{\alpha
}\right) ,...,st\left( x_{n},\mathcal{O}_{\alpha }\right) \right\} \cup
\left( \bigcup \mathcal{S}\right) $ covers $Y_{\alpha }.$\newline
Hence, each $Y_{\alpha }$ is covered by a finite subfamily of $\mathcal{M}$.
Thus, $X$ is covered by a subcover of $\mathcal{M}$ of cardinality at most $%
\kappa .$ This contradiction enables us to claim that $X$ is compact.\bigskip
\end{proof}

Since $\Delta \left( X\right) =w\left( X\right) $, for a compact Hausdorff
space $X$ (for example, in \cite[7.6. Corollary]{Hodel}), we can assert the
following.\medskip 

\begin{corollary}
If $X$ is a Hausdorff initially $\kappa $-compact space with a $G_{\kappa }$%
-diagonal, we have $w\left( X\right) \leq \kappa .\bigskip $
\end{corollary}

Recall that \textit{the pseudocharacter} \textit{of} $X$ \textit{at a subset}
$A$, denoted by $\psi (A,X)$, is defined as the smallest cardinal number of
the form $\left\vert \mathcal{U}\right\vert ,$ where $\mathcal{U}$ is a
family of open subsets of $X$ such that $\bigcap \mathcal{U}=A$. If $%
A=\left\{ x\right\} $ is a singleton, then we write $\psi \left( x,X\right) $
instead of $\psi \left( \left\{ x\right\} ,X\right) .$ \textit{The
pseudocharacter} \textit{of a space} $X$ is defined to be $\psi \left(
X\right) =\sup \left\{ \psi \left( x,X\right) :x\in X\right\} $. Evidently,
the diagonal number $\Delta \left( X\right) $ of a space $X$ is equal to the
pseudocharacter of its square $X\times X$ at its diagonal $\Delta
_{X}=\left\{ \left( x,x\right) :x\in X\right\} .$\medskip

It is well known that, if $G$ is a topological group, we have $\Delta \left(
G\right) =\psi \left( G\right) .$ So, Theorem 2.1 enables us to claim the
following.\medskip

\begin{corollary}
If $G$ is an initially $\kappa $-compact topological group with $\psi \left(
G\right) \leq \kappa ,$ then $G$ is compact.
\end{corollary}

\end{document}